\documentclass[12pt]{amsart}

\usepackage{amsmath,amssymb,amsrefs,enumerate,esint,pstricks, pst-plot}

\newcommand{\tphi}{\tilde{\phi}}
\newcommand{\tpsi}{\tilde{\psi}}

\newtheorem{thm}{Theorem}

\begin{document}

\title[Wave constant-mean-curvature equation]{Absence of Self-Similar
  Blow-up and Local Well-posedness for the Constant Mean-Curvature
  Wave Equation}
\author{Sagun Chanillo}
\address{Department of Mathematics, Rutgers, the State University of New Jersey, 110 Frelinghuysen Road, Piscataway, NJ 08854}
\email{chanillo@math.rutgers.edu}
\author{Po-Lam Yung}
\address{Department of Mathematics, the Chinese University of Hong Kong, Shatin, Hong Kong}
\email{plyung@math.cuhk.edu.hk}

\begin{abstract}
In this note, we consider the constant-mean-curvature wave equation in $(1+2)$-dimensions. We show that it does not admit any self-similar blow-up. We also remark that the equation is locally well-posed for initial data in $\dot{H}^{3/2}$.
\end{abstract}

\maketitle

\section{Introduction}

The study of nonlinear wave equations has attracted a lot of interest in recent years. One equation of interest is called the wave map. Recall that a wave map from the Minkowski space $\mathbb{R}^{n} \times \mathbb{R}$ into a $m$-dimensional Riemannian manifold $M$ is a map $u \colon \mathbb{R}^n \times \mathbb{R} \to M$, which in local coordinates satisfies
\begin{equation} \label{eq:wavemap}
(- \partial_t^2 + \Delta) u^a = -\Gamma^a_{bc}(u) (\partial^{\alpha} u^b) (\partial_{\alpha} u^c).
\end{equation}
Here $\{u^a\}_{a=1}^m$ are the local coordinates of the point $u(x,t)$ in $M$, $\Gamma^a_{bc}$ are the Christoffel symbols of $M$ in the corresponding local coordinates, and we use Einstein's summation notation, so that repeated indices on the right hand side are summed (the sum of $\alpha$ is over all $\alpha = 0,1,\dots,n$; we point out that the indices $\alpha$ are raised using the Minkowski metric on $\mathbb{R}^{n} \times \mathbb{R}$). Thus wave maps are just wave analogs of harmonic maps, which has been intensely studied for their connections to geometry and physics. 

The equation (\ref{eq:wavemap}) is critical for initial data in $\dot{H}^{\frac{n}{2}} \times \dot{H}^{\frac{n}{2}-1}$. Well-posedness for initial data slightly above, or at critical regularity, has been studied by many authors; we recall some of these shortly. One important aspect that arise in the analysis of wave maps is the \emph{null structure} of the non-linearity of the wave map equations. This has long been recognized, since the pioneering work of Klainerman and Machedon \cite{MR1231427}. The simplest prototypes of null forms can be displayed as follows:
\begin{align}
Q_{00}(u,v) &= -(\partial_t u)(\partial_t v) + \nabla_x u \cdot \nabla_x v \notag \\ 
Q_{ij}(u,v) &= (\partial_{x_i} u) (\partial_{x_j} v) - (\partial_{x_j} u) (\partial_{x_i} v), \quad i,j \in \{1,\dots,n\} \label{Qij} \\
Q_{0j}(u,v) &= (\partial_t u)(\partial_{x_j}v) - (\partial_{x_j} u) (\partial_t v), \quad i \in \{1,\dots,n\}. \notag
\end{align}
Null forms of type $Q_{00}$ arise in the study of wave maps. Using this null structure and the wave Sobolev $X^{s,b}$ spaces, Klainerman-Machedon \cite{MR1381973}, \cite{MR1446618} and Klainerman-Selberg \cite{MR1452172}, \cite{MR1901147} established subcritical local well-posedness for initial data in $H^s \times H^{s-1}$, $s > n/2$. At the critical regularity, one needs to bring in more geometric structures of the wave map equations, and write the equation in appropriate gauges. Tataru \cite{MR2130618} proved that if the target $M$ is uniformly isometrically embedded into some Euclidean space, then one has global existence for smooth initial data that has small $\dot{H}^{n/2} \times \dot{H}^{n/2-1}$ norm, with control of the $L^{\infty}_t \dot{H}^{n/2}_x$ norm of the solution. The case where the target $M$ is a sphere was obtained earlier in Tao \cite{MR1820329}, \cite{MR1869874} by introducing what is called a microlocal gauge, and the case where $n = 2$ and $M =$ the hyperbolic plane was also in Krieger \cite{MR2094472}. Krieger-Schlag \cite{MR2895939} later extended this latter result to the case of large initial data. Furthermore, Sterbenz-Tataru \cite{MR2657818}, \cite{MR2657817} proved that one has global existence and regularity for wave maps from $\mathbb{R}^{n}\times \mathbb{R}$ into $M$ if the energy of initial data is smaller than the energy of any nontrivial harmonic map $\mathbb{R}^n \to M$. 

On the other hand, Krieger-Schlag-Tataru \cite{MR2372807} proved the existence of equivariant finite time blow-up solutions for the wave map problem from $\mathbb{R}^{1+2}$ to $\mathbb{S}^2$. Later, Rodnianski-Sterbenz \cite{MR2680419} and Raphael-Rodnianski \cite{MR2929728} considered corotational wave maps from 2 space dimensions into the sphere $\mathbb{S}^2$ with initial data in $\dot{H}^1 \times L^2$, and exhibited an open subset of initial data in any given homotopy class that leads to finite time blow up. In fact they have also obtained a rather precise blow-up rate of the blow-ups they constructed.

In this paper, we study another system of wave equations, with a different null-structure. It is the wave analog of the equation that prescribes constant mean curvature. We call it the wave constant-mean-curvature equation (or wave CMC for short). The equation will be for maps $u$ which are defined on a domain in the (2+1)-dimensional Minkowski space $\mathbb{R}^{1+2}$, on which we use coordinates $(t,x,y)$, and which maps into $\mathbb{R}^3$. In fact, a map $u \colon [0,T] \times \mathbb{R}^2 \to \mathbb{R}^3$ is said to be a solution of the wave CMC, if on $ [0,T] \times \mathbb{R}^2 $ we have
\begin{equation} \label{eq:waveCMC}
(-\partial_t^2 + \Delta) u = 2 u_x \wedge u_y 
\end{equation}
where $\Delta$ is the Laplacian on $\mathbb{R}^2$ acting componentwise on the three components of $u$, and $u_x \wedge u_y$ is the cross product of the two vectors $u_x$ and $u_y$ in $\mathbb{R}^3$ (hence the null form $Q_{12}$ arises here). The stationary (elliptic) analog of this equation is
\begin{equation} \label{eq:CMC}
\Delta u = 2 u_x \wedge u_y.
\end{equation}
This is an interesting equation because if $u$ solves $\Delta u = 2H
u_x \wedge u_y$ for some function $H$ on $\mathbb{R}^2$ and satisfies
the conformal conditions $|u_x| = |u_y| = 1$ and $u_x \cdot u_y = 0$
everywhere, then the image of $u$ is a surface with mean curvature $H$
in $\mathbb{R}^3$. (\ref{eq:CMC}) is the special case of the equation
$\Delta u = 2H u_x \wedge u_y$ when $H \equiv 1$; hence the name wave
CMC for (\ref{eq:waveCMC}). We recall that (\ref{eq:CMC}) has been
studied by many authors in connection to semi-linear elliptic systems
of partial differential equations; see e.g. Hildebrandt
\cite{MR0256276}, Wente \cite{MR0243467}, \cite{MR0374673},
Brezis-Coron \cite{MR733715}, \cite{MR784102}, Struwe \cite{MR823116},
\cite{MR774369}, \cite{MR926524}, Chanillo-Malchiodi \cite{MR2154671}
and Caldiroli-Musina \cite{MR2221202}. In particular, bubbling phenomena for (\ref{eq:CMC}) was first studied by Brezis-Coron \cite{MR784102}, and a more refined bubbling analysis was done in Chanillo-Malchiodi \cite{MR2154671}. 

We also remind the reader that in Chanillo-Yung \cite{MR3010056}, Theorem 7, it is shown that (\ref{eq:waveCMC}) blows up in finite time if the initial energy exceeds that of the primary bubble of \cite{MR784102}, the Riemann sphere with winding number one. Thus we are naturally led to understanding possible natures of the blow-up in \cite{MR3010056}. 

Another motivation for us in studying (\ref{eq:waveCMC}) comes from the study of the energy-critical focusing semilinear wave equation. It is an equation for a scalar-valued function $u$ on $\mathbb{R}^{1+n}$, $n \geq 3$, given by
\begin{equation} \label{eq:waveYamabe}
(\partial_t^2 - \Delta) u = |u|^{4/(n-2)} u,
\end{equation}
and it is also sometimes called the wave Yamabe equation, since it is the wave analog of the Yamabe equation
$$
-\Delta u = |u|^{4/(n-2)} u
$$
in conformal geometry.  Kenig and Merle have developed in \cite{MR2461508} a concentration compactness-rigidity approach to establish global well-posedness of (\ref{eq:waveYamabe}) under a suitable class of initial data. A fundamental step in their work is the following result:
\begin{thm}[Kenig-Merle \cite{MR2461508}] \label{thm:KM}
Suppose $3 \leq n \leq 5$, and $(u_0,u_1) \in \dot{H}^1 \times L^2$ on $\mathbb{R}^n$ is such that
$$
\int_{\mathbb{R}^n} \nabla u_0 u_1 = 0, \quad \|\nabla u_0\|_{L^2} < \|\nabla W\|_{L^2}, \quad \text{and} \quad E(u_0,u_1) < E(W,0),
$$
where $W(x) = \left(1 + \frac{|x|^2}{n(n-2)} \right)^{-\frac{n-2}{2}}$ is the `groundstate' stationary solution to (\ref{eq:waveYamabe}), and
$$
E(u_0,u_1):= \int_{\mathbb{R}^n} \frac{1}{2} (|\nabla u_0|^2 + |u_1|^2) - \frac{n-2}{2n} |u|^{\frac{2n}{n-2}} 
$$
is the conserved energy of (\ref{eq:waveYamabe}). Suppose also that $u \colon [-1,0) \times \mathbb{R}^n \to \mathbb{R}$ is a solution to (\ref{eq:waveYamabe}), with initial data $\left. u \right|_{t = -1} = u_0$, $\left. \partial_t u \right|_{t = -1} = u_1$. For $t \in [-1,0)$, let 
$$
U_t(x) = (-t)^{n/2} (\nabla u)(t, -tx), \quad V_t(x) = (-t)^{n/2} (\partial_t u)(t, -tx).
$$
If the solution $u(t,x)$ does not extend beyond $t = 0$, then the set
$$
\left\{ (U_t,V_t) \colon t \in [-1,0) \right\} 
$$
cannot have compact closure in $\dot{H}^1 \times L^2$.
\end{thm}
Theorem~\ref{thm:KM} occupies a substantial portion of Section 6 of \cite{MR2461508}, leading to a unique continuation problem for a degenerate elliptic equation, Proposition 6.12 in \cite{MR2461508}. In particular, this rules out the existence of self-similar solutions to (\ref{eq:waveYamabe}).

Our future goal is to also apply concentration compactness-rigidity method to study equation (\ref{eq:waveCMC}). As a first step, we show in this paper that the wave CMC (\ref{eq:waveCMC}) does not admit self-similar blow-ups. More precisely, we prove in Section~\ref{sect:ssblowup} the following result:

\begin{thm} \label{thm:ssblowup}
Suppose $v \in C^2(\mathbb{R}^2, \mathbb{R}^3)$ is such that
\begin{equation} \label{eq:sbuform}
u(x,y,t) = v(\frac{x}{t},\frac{y}{t})
\end{equation} 
is a solution to the wave CMC
$$
(-\partial_t^2 + \Delta) u = 2 u_x \wedge u_y \quad \text{on $\{t > 0\}$}.
$$
Then $v$ is constant on $\mathbb{D}$, where $\mathbb{D}$ is the unit disc on $\mathbb{R}^2$.
\end{thm}

Note that by scaling and dimensional considerations, the self-similar blow-up can only be of the type (\ref{eq:sbuform}). 

The analogue of this statement for wave maps have been long known; see e.g. Chapter 7.5 of the monograph of Shatah-Struwe \cite{MR1674843}. 

Our approach to Theorem~\ref{thm:ssblowup} is inspired by that in \cite{MR1674843}. Using a Pohozhaev type identity from \cite{MR2154671}, we are also led to a unique continuation problem for a degenerate elliptic equation in self-similar coordinates. But since we are in dimension two, we may apply the uniformization theorem like in \cite{MR1674843}, and reduce matters to a unique continuation theorem of Hartman-Wintner \cite{MR2154671}, \cite{MR0058082}.

We are making the qualitative assumption $v \in C^2$ in the theorem, only to justify various integration by parts arguments in our proof.

We now turn our attention to local well-posedness of the initial value problem for the wave CMC (\ref{eq:waveCMC}). 
An easy scaling argument reveals that the wave CMC (\ref{eq:waveCMC}) is critical for initial data in $\dot{H}^1 \times L^2$. Well-posedness at such sharp regularity seems way out of reach at this point, because of the lack of sufficiently powerful Strichartz estimates in (2+1) dimensions. On the other hand, we take this chance to point out that the wave CMC (\ref{eq:waveCMC}) is locally well-posed, for initial data in $\dot{H}^{3/2} \times \dot{H}^{1/2}$; this is basically an observation that dates back to Klainerman-Machedon \cite{MR1231427} (see Remark 3 in \cite{MR1231427}).

\begin{thm} \label{thm:3halfwp}
Given any $K > 0$, there exists a small $T > 0$, and a constant $A > 0$ (both depending only on $K$) such that for any initial data $(u_0,u_1) \in \dot{H}^{3/2} \times \dot{H}^{1/2}$ with 
$$\|u_0\|_{\dot{H}^{3/2}} + \|u_1\|_{\dot{H}^{1/2}} \leq K,$$ 
the Cauchy problem
$$
(-\partial_t^2 + \Delta) u = 2 u_x \wedge u_y, \quad \left.u \right|_{t=0} = u_0, \quad \left.\partial_t u\right|_{t=0} = u_1
$$
has a unique solution $u$ on $[0,T]$ with
$$
u \in C^0_{[0,T]} \dot{H}^{3/2}, \quad \partial_t u \in C^0_{[0,T]} \dot{H}^{1/2}, \quad \|u_x \wedge u_y\|_{L^2_{[0,T]} \dot{H}^{1/2}} \leq A,
$$
in the sense that $u$ solves the following integral equation for $t \in [0,T]$:
$$
u(t) = \cos(t \sqrt{-\Delta}) u_0 + \frac{\sin(t \sqrt{-\Delta})}{\sqrt{-\Delta}} u_1 + \int_0^t \frac{\sin((t-s)\sqrt{-\Delta})}{\sqrt{-\Delta}} 2u_x \wedge u_y(s) ds.
$$
Furthermore, the map $$(u_0,u_1) \mapsto (u,\partial_t u)$$
$$\dot{H}^{3/2} \times \dot{H}^{1/2} \to C^0_{[0,T]} \dot{H}^{3/2} \times C^0_{[0,T]} \dot{H}^{1/2}$$ is continuous on the set $\{(u_0,u_1) \colon \|u_0\|_{\dot{H}^{3/2}} + \|u_1\|_{\dot{H}^{1/2}} \leq K\}$.
\end{thm}

Here $C^0_{[0,T]} \dot{H}^s$ is the set of maps $u \colon [0,T] \times \mathbb{R}^2 \to \mathbb{R}^3$ with
$$
\|u\|_{C^0_{[0,T]} \dot{H}^s} := \sup_{t \in [0,T]} \left( \int_{\mathbb{R}^2} |(-\Delta_{x,y})^{s/2} u(t,x,y)|^2 dx dy \right)^{1/2} < \infty,
$$
and similarly one defines $L^p_{[0,T]} \dot{H}^s$. We provide a proof of Theorem~\ref{thm:3halfwp}  in Section~\ref{sect:3halfwp}, for the convenience of the reader.

\vspace{0.05 in}

\noindent{\textbf{Ackonwledgements}}. S.C. and P.-L.Y. were supported by NSF grant DMS-1201474. P.-L.Y. was also supported by a Titchmarsh Fellowship at the University of Oxford, a junior research fellowship at St. Hilda's College, and a direct allocation grant from the Chinese University of Hong Kong. S.C. wishes to thank Carlos Kenig for several enlightening conversations related to the contents of this paper. 

\section{Non-existence of self-similar blow-ups} \label{sect:ssblowup}

\begin{proof}[Proof of Theorem~\ref{thm:ssblowup}]
By finite speed of propagation, we may assume that $v$ is defined only on $\overline{\mathbb{D}}$, and consider the solution $u$ only in the light cone $$\Gamma:=\{(x,y,t)\colon \sqrt{x^2+y^2} \leq t\}.$$ We will also forget about the values of $v$ outside $\overline{\mathbb{D}}$.

Now we introduce self-similar variables
$$
\tau = \sqrt{t^2 - x^2 - y^2}, \quad \rho = \frac{\sqrt{x^2 + y^2}}{t}, \quad \rho e^{i\theta} = \frac{x + i y}{t}
$$
which is a re-parametrization of $\Gamma$.
($\tau$ is well-defined since we are now in the light cone $\Gamma$.) We also write $v = v(\rho,\theta)$ for $\rho \in [0,1]$, $\theta \in [0,2\pi]$. Then the Minkowski metric 
$$ds^2 = -dt^2 + dx^2 + dy^2$$ on $\Gamma$ becomes
$$
ds^2 = -d\tau^2 + \frac{\tau^2 d\rho^2}{(1-\rho^2)^2} + \frac{\tau^2 \rho^2 d\theta^2}{1-\rho^2}
$$
in the new $(\tau,\rho,\theta)$ coordinate system, i.e. 
\begin{equation} \label{metric}
g_{\tau \tau} = -1, \quad g_{\rho\rho} = \frac{\tau^2}{(1-\rho^2)^2} , \quad  \text{and} \quad g_{\theta\theta} = \frac{\tau^2 \rho^2}{1-\rho^2}.
\end{equation} 
In fact,
$$
\tau^2 = t^2 - (x^2 + y^2) = t^2 - \rho^2 t^2 = (1-\rho^2)t^2,
$$
so
$$
t = \frac{\tau}{\sqrt{1-\rho^2}}.
$$
Also,  $x = \rho t \cos \theta$, $y =  \rho t \sin \theta$, so
$$ x= \frac{\tau \rho}{\sqrt{1-\rho^2}} \cos \theta, \quad y  = \frac{\tau \rho}{\sqrt{1-\rho^2}} \sin \theta.
$$
It follows that
\begin{align*}
dt &= \frac{1}{\sqrt{1-\rho^2}} d\tau + \frac{\tau \rho}{(1-\rho^2)^{3/2}} d\rho \\
dx &= \frac{\rho}{\sqrt{1-\rho^2}} \cos \theta d\tau + \frac{\tau}{(1-\rho^2)^{3/2}} \cos \theta d\rho - \frac{\tau \rho}{\sqrt{1-\rho^2}} \sin \theta d\theta \\
dy &= \frac{\rho}{\sqrt{1-\rho^2}} \sin \theta d\tau + \frac{\tau}{(1-\rho^2)^{3/2}} \sin \theta d\rho + \frac{\tau \rho}{\sqrt{1-\rho^2}} \cos \theta d\theta,
\end{align*}
and (\ref{metric}) follows.
From (\ref{metric}), we have 
$$\sqrt{|\det(g_{ij})|} = \frac{\tau^2 \rho}{(1-\rho^2)^{3/2}},$$ 
$$g^{\tau\tau} = -1, \quad g^{\rho\rho} = \frac{(1-\rho^2)^2}{\tau^2} , \quad  \text{and} \quad g^{\theta\theta} = \frac{1-\rho^2}{\tau^2 \rho^2}.$$ It follows that the wave operator $-\partial_t^2 + \Delta_{x,y}$ becomes
\begin{align*}
& \frac{(1-\rho^2)^{3/2}}{\tau^2 \rho} \left[
\partial_{\tau} \left( -\frac{\tau^2 \rho}{(1-\rho^2)^{3/2}} \partial_{\tau} \cdot \right)  + 
\partial_{\rho} \left( \frac{\tau^2 \rho}{(1-\rho^2)^{3/2}} \frac{(1-\rho^2)^2}{\tau^2} \partial_{\rho} \cdot \right) \right. \\
& \qquad \qquad \qquad\left. + \partial_{\theta} \left( \frac{\tau^2 \rho}{(1-\rho^2)^{3/2}} \frac{1-\rho^2}{\tau^2 \rho^2} \partial_{\theta} \cdot \right) \right] \\
= & -\frac{(1-\rho^2)^{3/2}}{\tau^2 \rho} \left[
\partial_{\tau} \left( \frac{\tau^2 \rho}{(1-\rho^2)^{3/2}} \partial_{\tau} \cdot \right)  - 
\partial_{\rho} \left( \rho \sqrt{1-\rho^2} \partial_{\rho} \cdot \right)  - 
\left( \frac{1}{\rho \sqrt{1-\rho^2}} \partial^2_{\theta} \cdot \right) \right] 
\end{align*}
Furthermore,
\begin{align*}
\partial_x 
&= \frac{\partial \tau}{\partial x} \partial_{\tau} +  \frac{\partial \rho}{\partial x} \partial_{\rho} +  \frac{\partial \theta}{\partial x} \partial_{\theta} \\
&= -\frac{x}{\sqrt{t^2-x^2-y^2}} \partial_{\tau} +  \frac{x}{t \sqrt{x^2 + y^2}} \partial_{\rho} - \frac{y}{x^2 + y^2} \partial_{\theta} \\
&= -\frac{\rho \cos \theta}{\sqrt{1-\rho^2}} \partial_{\tau} +  \frac{\sqrt{1-\rho^2} \cos \theta}{\tau} \partial_{\rho} - \frac{\sqrt{1-\rho^2} \sin \theta}{\tau \rho} \partial_{\theta}
\end{align*}
and similarly
\begin{align*}
\partial_y &= \frac{\partial \tau}{\partial y} \partial_{\tau} +  \frac{\partial \rho}{\partial y} \partial_{\rho} +  \frac{\partial \theta}{\partial y} \partial_{\theta} \\
&= -\frac{y}{\sqrt{t^2-x^2-y^2}} \partial_{\tau} +  \frac{y}{t \sqrt{x^2 + y^2}} \partial_{\rho} + \frac{x}{x^2 + y^2} \partial_{\theta} \\
&= -\frac{\rho \sin \theta}{\sqrt{1-\rho^2}} \partial_{\tau} +  \frac{\sqrt{1-\rho^2} \sin \theta}{\tau} \partial_{\rho} + \frac{\sqrt{1-\rho^2} \cos \theta}{\tau \rho} \partial_{\theta}.
\end{align*}
Applying these to $v(\rho,\theta)$, and noting that it is independent of $\tau$, we see that the wave CMC for $u$ becomes
$$
\frac{(1-\rho^2)^{3/2}}{\tau^2 \rho} \left[ 
\partial_{\rho} \left( \rho \sqrt{1-\rho^2} \partial_{\rho} v \right)  +
\frac{1}{\rho \sqrt{1-\rho^2}} \partial^2_{\theta} v  \right] 
= 2 \frac{1-\rho^2}{\tau^2 \rho} v_{\rho} \wedge v_{\theta},
$$
i.e.
\begin{equation} \label{vselfsimilarcoord}
\rho \sqrt{1-\rho^2} \partial_{\rho}  \left( \rho \sqrt{1-\rho^2} \partial_{\rho} v \right) +  \partial^2_{\theta} v 
= 2\rho v_{\rho} \wedge v_{\theta}.
\end{equation}
Now take the dot product of both sides with $v_{\rho}$. Then the right hand side vanishes, and we get
$$
\frac{1}{2} \partial_{\rho}  \left| \rho \sqrt{1-\rho^2} v_{\rho} \right|^2 + v_{\rho}  \cdot \partial^2_{\theta} v = 0.
$$
Integrating this in $\theta$, and integrating by parts in the last term, we get
$$
\frac{d}{d\rho} \int_0^{2\pi} \left[ \rho^2 (1-\rho^2) |v_{\rho}|^2 - |v_{\theta} |^2 \right] d\theta = 0,
$$
i.e. 
\begin{equation} \label{vthetaconst}
\int_0^{2\pi} \left[ \rho^2 (1-\rho^2)  |v_{\rho}|^2 - |v_{\theta} |^2  \right] d\theta
\end{equation}
is a constant independent of $\rho$. Letting $\rho \to 0$, we see that this constant is zero (note that $v_{\theta} = O(\rho)$ as $\rho \to 0$, if $v$ is differentiable at $0$). So when $\rho = 1$, the integral (\ref{vthetaconst}) is equal to 0. It follows that
$$
\int_0^{2\pi} |v_{\theta}|^2 d\theta = 0 \qquad \text{when $\rho = 1$},
$$
i.e. $v$ is a constant on $\partial \mathbb{D}$. Note that the wave CMC has no zeroth order term. Hence we may subtract a constant from $v$, to make 
$$
v = 0 \quad \text{on $\partial \mathbb{D}$},
$$
and we will do so from now on.

Now introduce a new variable $$\sigma = \exp \left( -\int_{\rho}^1 \frac{ds}{s \sqrt{1-s^2}} \right) = \frac{\rho}{1+\sqrt{1-\rho^2}}$$ such that
$$
\sigma \frac{\partial}{\partial \sigma} = \rho \sqrt{1-\rho^2} \frac{\partial}{\partial \rho}.
$$
(Note that as $\rho$ varies between $[0,1]$, $\sigma$ also varies between $[0,1]$.) Then equation (\ref{vselfsimilarcoord}) for $v$ becomes
$$
\sigma \partial_{\sigma} (\sigma \partial_{\sigma} v) + \partial^2_{\theta} v = 2 \frac{\sigma}{\sqrt{1-\rho^2}} v_{\sigma} \wedge v_{\theta},
$$
i.e.
\begin{equation} \label{eq:vsigma}
v_{\sigma \sigma} + \frac{1}{\sigma} v_{\sigma} + \frac{1}{\sigma^2} v_{\theta \theta} =  2 \frac{1}{\sigma \sqrt{1-\rho^2}} v_{\sigma} \wedge v_{\theta}.
\end{equation}
Let $z = \sigma e^{i\theta}$. Then on the left hand side we have then the flat Laplacian on the $z$-plane. On the right hand side we have something normal to the surface parametrized by $v$. So we can apply the strategy of Chanillo-Malchiodi \cite{MR2154671}, Lemma 3.1. More precisely, we take the dot product of both sides of the equation with $\sigma v_{\sigma}$. Then
$$
\sigma v_{\sigma} \cdot [v_{\sigma \sigma} + \frac{1}{\sigma} v_{\sigma} + \frac{1}{\sigma^2} v_{\theta \theta}] =  0
$$
on the unit disc $\mathbb{D}$ in the $z$-plane. We integrate over $\mathbb{D}$ using polar coordinates: 
\begin{align*}
0 
&= \int_0^1 \int_0^{2\pi} \sigma v_{\sigma} \cdot [v_{\sigma \sigma} + \frac{1}{\sigma} v_{\sigma} + \frac{1}{\sigma^2} v_{\theta \theta}] \sigma  d\theta d\sigma\\
&= \int_0^1 \int_0^{2\pi} \frac{\partial}{\partial \sigma} \left(\frac{1}{2} \sigma^2 |v_{\sigma}|^2 - \frac{1}{2} |v_{\theta}|^2 \right) + \frac{\partial}{\partial \theta} (v_{\sigma} \cdot v_{\theta})  d\theta d\sigma \\
&=  \int_0^{2\pi} \left[ \frac{1}{2} \sigma^2 |v_{\sigma}|^2 - \frac{1}{2} |v_{\theta}|^2 \right]_{\sigma = 0}^{\sigma = 1} d\theta \\
&=  \int_0^{2\pi} \left[ \frac{1}{2} |v_{\sigma}|^2 \right]_{\sigma = 1} d\theta
\end{align*}
(the last equality following from $v_{\theta} = 0$ when $\sigma = 1$ (note $\sigma = 1$ if and only if $\rho = 1$), and that $$v_{\theta} = \rho (v_y \cos \theta - v_x \sin \theta) = O(\sigma |\nabla_{x,y} v|) \to 0$$ as $\sigma \to 0$). Hence
$$
v_{\sigma} = 0 \quad \text{on $\{\sigma = 1\}$}.
$$
Now we extend $v$ so that $v = 0$ when $\sigma > 1$ . Then from the above, $v(z) \in C^1(\mathbb{C} \setminus \{0\})$. We can then apply the unique continuation technique of Hartman-Wintner. More precisely, from Theorem 2 of \cite{MR0058082}, we have:
\begin{thm}[Hartman-Wintner \cite{MR0058082}] \label{thm:HW}
Suppose $r \in \mathbb{N}$, $U$ is an open set in $\mathbb{C}$, and $v \in C^1(U,\mathbb{R}^r)$, and there exists  continuous (matrix-valued) functions $d, e, f$ on $U$ such that 
\begin{equation} \label{eq:HWid}
\int_{\partial \Omega} g(z) \frac{\partial v}{\partial z} dz = \int_{\Omega} g(z) \left[d(z) \frac{\partial v}{\partial z} + e(z) \frac{\partial v}{\partial \overline{z}} + f(z) v(z) \right] dz \wedge d\overline{z}
\end{equation}
for all piecewise smooth relatively compact domains $\Omega$ of $U$ and all holomorphic functions $g$ on $\overline{\Omega}$. If there exists a point $z_0 \in U$ such that 
\begin{equation} \label{eq:HWorder}
\lim_{z \to z_0} \frac{v(z)}{(z-z_0)^n} = 0 \quad \text{for all $n \in \mathbb{N}$,}
\end{equation}
then $$v \equiv 0 \quad \text{on $U$}.$$
\end{thm}

We verify that the conditions of the above theorem is met, when $r = 3$ and $U = \mathbb{C} \setminus \{0\}$: First, since $v \in C^1(\mathbb{C}\setminus\{0\})$, and is supported in $\overline{\mathbb{D}}$, we have, for any piecewise smooth relatively compact domain $\Omega \subset \mathbb{C} \setminus \{0\}$, that
\begin{align*}
\int_{\partial \Omega} g(z) \frac{\partial v}{\partial z} dz 
=& \int_{\partial (\Omega \cap \mathbb{D})} g(z) \frac{\partial v}{\partial z} dz \\
=& \int_{\Omega \cap \mathbb{D}} \frac{\partial}{\partial \overline{z}} \left( g(z) \frac{\partial v}{\partial z} \right) d\overline{z} \wedge dz \\
=& \int_{\Omega \cap \mathbb{D}}  g(z) \frac{\partial^2 v}{\partial z \partial \overline{z}} d\overline{z} \wedge dz.
\end{align*}
But by (\ref{eq:vsigma}), on $\mathbb{D}$ we have
$$
\frac{1}{4} \frac{\partial^2 v}{\partial z \partial \overline{z}} = v_{\sigma \sigma} + \frac{1}{\sigma} v_{\sigma} + \frac{1}{\sigma^2} v_{\theta \theta} = 
 2 \frac{1}{\sqrt{1-\rho^2}} v_{\sigma} \wedge \frac{1}{\sigma} v_{\theta},
$$
and
$$
\frac{1}{\sqrt{1-\rho^2}} v_{\sigma} = \frac{1}{\sqrt{1-\rho^2}} \frac{\rho \sqrt{1-\rho^2}}{\sigma} v_{\rho} = \frac{\rho}{\sigma} v_{\rho} = \frac{1}{1+\sqrt{1-\rho^2}} v_{\rho}
$$
is continuous up to $\{\sigma = 1\}$. 
Also, $\frac{1}{\sigma} v_{\theta}$ is a linear combination of $\frac{\partial v}{\partial z}$ and $\frac{\partial v}{\partial \overline{z}}$ on $\mathbb{C} \setminus \{0\}$. Hence one can find continuous (matrix-valued) functions $d$ and $e$ on $\overline{\mathbb{D}} \setminus \{0\}$, such that
$$
\frac{\partial^2 v}{\partial z \partial \overline{z}} 
= d(z) \frac{\partial v}{\partial z} + e(z) \frac{\partial v}{\partial \overline{z}}
$$
on $\mathbb{D} \setminus \{0\}$. Extending $d$ and $e$ continuously to $\mathbb{C}$, and using that $v$ vanishes outside $\overline{\mathbb{D}}$, we see that (\ref{eq:HWid}) is satisfied with $f = 0$. Also, (\ref{eq:HWorder}) is satisfied at any $z_0 \in \mathbb{C} \setminus \overline{\mathbb{D}}$. Hence Theorem~\ref{thm:HW} implies that $v \equiv 0$ on $\mathbb{C} \setminus \{0\}$, which also implies $v(0) = 0$ by continuity. In particular, we have $v(z) = 0$ for all $|z| < 1$, i.e. $v(\rho,\theta) = 0$ whenever $\rho < 1$, as desired.

\end{proof}

\section{Local well-posedness in $\dot{H}^{3/2}$} \label{sect:3halfwp}

\begin{proof}[Proof of Theorem~\ref{thm:3halfwp}]
Let $\square = (-\partial_t^2 + \Delta)$ be the D'Alembertian on $\mathbb{R}^{1+2}$. Recall the null form $Q_{12}$ from (\ref{Qij}). Note that the wave CMC (\ref{eq:waveCMC}) is a system of equations, that can be written in the components $(u_1,u_2,u_3)$ of $u$ as
\begin{align*}
\square u_1 &= 2 Q_{12}(u_2,u_3) \\
\square u_2 &= 2 Q_{12}(u_3,u_1) \\
\square u_3 &= 2 Q_{12}(u_1,u_2).
\end{align*}
Also, as was observed in Klainerman-Machedon \cite{MR1231427}, we have the following estimates for the null form $Q_{12}$ on $\mathbb{R}^{1+2}$: if $f, g \in \dot{H}^{1/2}(\mathbb{R}^2)$, and
$$
\phi_{\pm} := \frac{e^{\pm i t \sqrt{-\Delta}}}{\sqrt{-\Delta}} f, \quad \psi_{\pm} := \frac{e^{\pm i t \sqrt{-\Delta}}}{\sqrt{-\Delta}} g,
$$
then
\begin{equation} \label{eq:derkeyQ12++}
\|(-\Delta)^{1/4} Q_{12} (\phi_+,\psi_+)\|_{L^2(\mathbb{R}^{1+2})} \leq C \|f\|_{\dot{H}^{1/2}} \|g\|_{\dot{H}^{1/2}}.
\end{equation}
(We briefly recall the proof of this at the end of the section, for the convenience of the reader.) The same continues to hold, if the signs $(+,+)$ on the left hand side are replaced by any of the choices $(-,-)$, $(+,-)$ and $(-,-)$. Thus if $u, v \colon [0,T] \times \mathbb{R}^2 \to \mathbb{R}^3$, with
$$
\square u = F, \quad \left. u \right|_{t=0} = u_0, \quad \left. \partial_t u \right|_{t=0} = u_1,
$$
$$
\square v = G, \quad \left. v \right|_{t=0} = v_0, \quad \left. \partial_t v \right|_{t=0} = v_1,
$$
then
\begin{align}
\|u_x \wedge v_y \|_{L^2_{[0,T]} \dot{H}^{1/2}}  \label{eq:wedgeest} 
\leq C & (\|u_0\|_{\dot{H}^{3/2}} + \|u_1\|_{\dot{H}^{1/2}} + \|F\|_{L^1_{[0,T]} \dot{H}^{1/2}}) \\
&  \cdot (\|v_0\|_{\dot{H}^{3/2}} + \|v_1\|_{\dot{H}^{1/2}} + \|G\|_{L^1_{[0,T]} \dot{H}^{1/2}}). \notag
\end{align}
Also, the standard energy estimate shows that
\begin{equation} \label{eq:energy32}
\|u\|_{C^0_{[0,T]} \dot{H}^{3/2}} + \|\partial_t u\|_{C^0_{[0,T]} \dot{H}^{1/2}} \leq 2( \|u_0\|_{\dot{H}^{3/2}} + \|u_1\|_{\dot{H}^{1/2}} + \|F\|_{L^1_{[0,T]} \dot{H}^{1/2}}).
\end{equation}

Now to prove the theorem, let $K$ be given, and set
$$A = \max \{2 K, 4 C K^2 \}$$ where from now on $C$ is the constant in (\ref{eq:wedgeest}).
Let $T > 0$ be sufficiently small, so that
\begin{equation} \label{eq:Tchoice1}
4T^{1/2} \leq \frac{1}{2},
\end{equation} 
\begin{equation} \label{eq:Tchoice2}
2T^{1/2}A \leq K,
\end{equation} 
and
\begin{equation} \label{eq:Tchoice3} 
2C T^{1/2} (K + 2T^{1/2} A) \leq \frac{1}{4}.
\end{equation}
To prove existence, we fix initial data $(u_0,u_1) \in \dot{H}^{3/2} \times \dot{H}^1$ with 
$$
\|u_0\|_{\dot{H}^{3/2} } + \|u_1\|_{\dot{H}^{1/2}} \leq K.
$$
Let
$$
u^{(0)} := 0,
$$
and for $k \geq 0$, let $u^{(k+1)}$ solve
$$
\square u^{(k+1)} = 2 u^{(k)}_x \wedge u^{(k)}_y, \quad \left. u^{(k+1)}\right|_{t = 0} = u_0, \quad \left. \partial_t u^{(k+1)} \right|_{t=0} = u_1.
$$
We will prove, by induction, that for all $k \geq 0$, 
\begin{equation} \label{eq:induct1}
\|u^{(k+1)} - u^{(k)}\|_{C^0_{[0,T]} \dot{H}^{3/2}} + \|\partial_t u^{(k+1)} - \partial_t u^{(k)}\|_{C^0_{[0,T]} \dot{H}^{1/2}} \leq \frac{A}{2^k}
\end{equation}
\begin{equation} \label{eq:induct2}
\|u^{(k+1)}_x \wedge u^{(k+1)}_y \|_{L^2_{[0,T]} \dot{H}^{1/2}} \leq A
\end{equation}
\begin{equation} \label{eq:induct3}
\|u^{(k+1)}_x \wedge u^{(k+1)}_y -u^{(k)}_x \wedge u^{(k)}_y \|_{L^2_{[0,T]} \dot{H}^{1/2}} \leq \frac{A}{2^k}.
\end{equation}
In fact, first consider the case $k=0$. 
Then from (\ref{eq:energy32}), we have
$$
\|u^{(1)}\|_{C^0_{[0,T]} \dot{H}^{3/2}} + \|\partial_t u^{(1)}\|_{C^0_{[0,T]} \dot{H}^{1/2}} \leq 2 (\|u_0\|_{\dot{H}^{3/2}} + \|u_1\|_{\dot{H}^1}) \leq 2K;
$$
from (\ref{eq:wedgeest}), we have
$$
\|u^{(1)}_x \wedge u^{(1)}_y\|_{L^2_{[0,T]} \dot{H}^{1/2}} \leq C (\|u_0\|_{\dot{H}^{3/2}} + \|u_1\|_{\dot{H}^{1/2}})^2 \leq CK^2.
$$
By our choice of $A$, this proves (\ref{eq:induct1}), (\ref{eq:induct2}) and (\ref{eq:induct3}) when $k = 0$.

Now suppose $k \geq 1$.  Then by (\ref{eq:energy32}),
\begin{align*}
& \|u^{(k+1)} - u^{(k)}\|_{C^0_{[0,T]} \dot{H}^{3/2}} + \|\partial_t u^{(k+1)} - \partial_t u^{(k)}\|_{C^0_{[0,T]} \dot{H}^{1/2}} \\
\leq & 2 \|2 u^{(k)}_x \wedge u^{(k)}_y - 2 u^{(k-1)}_x \wedge u^{(k-1)}_y\|_{L^1_{[0,T]} \dot{H}^{1/2}} \\
\leq & 4 T^{1/2} \frac{A}{2^{k-1}} \\
\leq & \frac{A}{2^k}.
\end{align*}
(The second-to-last inequality follows from (\ref{eq:induct3}) for $k-1$ in place of $k$, and the last inequality from (\ref{eq:Tchoice1}).) Also,
by (\ref{eq:wedgeest}),
\begin{align*}
& \|u^{(k+1)}_x \wedge u^{(k+1)}_y \|_{L^2_{[0,T]} \dot{H}^{1/2}} \\
\leq & C (\|u_0\|_{\dot{H}^{3/2}} + \|u_1\|_{\dot{H}^{1/2}} + \|2 u^{(k)}_x \wedge u^{(k)}_y \|_{L^1_{[0,T]} \dot{H}^{1/2}})^2 \\
\leq & C (K + 2 T^{1/2} \|u^{(k)}_x \wedge u^{(k)}_y \|_{L^2_{[0,T]} \dot{H}^{1/2}})^2 \\
\leq & C (K + 2 T^{1/2} A)^2 \leq  C(2K)^2 \leq A,
\end{align*}
(The second inequality follows from (\ref{eq:induct2}) with $k$ replaced by $k-1$, and the third inequality from (\ref{eq:Tchoice2}).) Finally, note that
$$
u^{(k+1)}_x \wedge u^{(k+1)}_y - u^{(k)}_x \wedge u^{(k)}_y 
= u^{(k+1)}_x \wedge (u^{(k+1)}-u^{(k)})_y  + (u^{(k+1)}-u^{(k)})_x \wedge u^{(k)}_y 
$$
But by (\ref{eq:wedgeest}),
\begin{align*}
&\|u^{(k+1)}_x \wedge (u^{(k+1)}-u^{(k)})_y \|_{L^2_{[0,T]} \dot{H}^{1/2}} \\
\leq & C (\|u_0\|_{\dot{H}^{3/2}} + \|u_1\|_{\dot{H}^{1/2}} + \|2 u^{(k)}_x \wedge u^{(k)}_y\|_{L^1_{[0,T]} \dot{H}^{1/2}}) \|2 u^{(k)}_x \wedge u^{(k)}_y - 2 u^{(k-1)}_x \wedge u^{(k-1)}_y \|_{L^1_{[0,T]} \dot{H}^{1/2}} \\
\leq & 2 C T^{1/2} (K + 2T^{1/2} A) \cdot \frac{A}{2^{k-1}} \\
\leq & \frac{A}{2^{k+1}}.
\end{align*}
(The second-to-last inequality follows from (\ref{eq:induct2}) and (\ref{eq:induct3}) with $k$ replaced by $k-1$, and the last inequality from (\ref{eq:Tchoice3}).) Similarly, one can show
\begin{align*}
\|(u^{(k+1)}-u^{(k)})_x \wedge u^{(k)}_y \|_{L^2_{[0,T]} \dot{H}^{1/2}} 
\leq & \frac{A}{2^{k+1}}.
\end{align*}
Together they prove (\ref{eq:induct3}). This completes our proof of (\ref{eq:induct1}), (\ref{eq:induct2}) and (\ref{eq:induct3}).

Let $X$ be the Banach space $$\{u \colon u \in C^0_{[0,T]} \dot{H}^{3/2}, \partial_t u \in C^0_{[0,T]} \dot{H}^{1/2}\}$$ with the natural norm. Then by (\ref{eq:induct1}), $u^{(k)}$ is Cauchy in $X$. We write $u$ for the limit of $u^{(k)}$ in $X$. Also, the sequence $u^{(k)}_x \wedge u^{(k)}_y$ is Cauchy in $L^2_{[0,T]} \dot{H}^{1/2}$, by (\ref{eq:induct3}). We write $F$ for the limit of $u^{(k)}_x \wedge u^{(k)}_y$ in $L^2_{[0,T]} \dot{H}^{1/2}$. In particular, 
\begin{equation} \label{eq:FL2bdd}
\|F\|_{L^2_{[0,T]}\dot{H}^{1/2}} \leq A,
\end{equation}
and $u^{(k)}_x \wedge u^{(k)}_y$ also converges to $F$ in $L^1_{[0,T]} \dot{H}^{1/2}$. Now for $t \in [0,T]$,
\begin{equation} \label{eq:ukDuhamel}
u^{(k+1)}(t) = \cos(t\sqrt{-\Delta}) u_0  + \frac{\sin(t \sqrt{-\Delta})}{\sqrt{-\Delta}} u_1 + \int_0^t \frac{\sin((t-s)\sqrt{-\Delta})}{\sqrt{-\Delta}} 2u^{(k)}_x \wedge u^{(k)}_y (s) ds.
\end{equation}
Passing to limit in $X$, we then see that
\begin{equation} \label{eq:uinftyDuhamel}
u(t) = \cos(t\sqrt{-\Delta}) u_0  + \frac{\sin(t \sqrt{-\Delta})}{\sqrt{-\Delta}} u_1 + \int_0^t \frac{\sin((t-s)\sqrt{-\Delta})}{\sqrt{-\Delta}} 2F(s) ds.
\end{equation}
(The convergence of the right hand side in $X$ is guaranteed by the energy estimate (\ref{eq:energy32}), and the convergence of $u^{(k)}_x \wedge u^{(k)}_y$ to $F$ in $L^1_{[0,T]} \dot{H}^{1/2}$.) On the other hand, we claim that $u^{(k)}_x \wedge u^{(k)}_y$ converges to $u_x \wedge u_y$ in $L^2_{[0,T]} \dot{H}^{1/2}$. Assuming the claim for the moment, we then see that $u_x \wedge u_y = F$ on $[0,T] \times \mathbb{R}^2$, so (\ref{eq:uinftyDuhamel}) becomes
$$
u(t) = \cos(t\sqrt{-\Delta}) u_0  + \frac{\sin(t \sqrt{-\Delta})}{\sqrt{-\Delta}} u_1 + \int_0^t \frac{\sin((t-s)\sqrt{-\Delta})}{\sqrt{-\Delta}} 2u_x \wedge u_y(s) ds,
$$
and  (\ref{eq:FL2bdd}) implies
$$
\|u_x \wedge u_y \|_{L^2_{[0,T]} \dot{H}^{1/2}} \leq A,
$$
as desired. So we now move on to prove the claim.

To do so, note that 
$$
u^{(k)}_x \wedge u^{(k)}_y - u_x \wedge u_y
= u^{(k)}_x \wedge (u^{(k)}-u)_y  +  (u^{(k)} - u)_x \wedge u_y,
$$
so by (\ref{eq:wedgeest}) and (\ref{eq:uinftyDuhamel}),
\begin{align*}
&\|u^{(k)}_x \wedge u^{(k)}_y - u_x \wedge u_y\|_{L^2_{[0,T]} \dot{H}^{1/2}} \\
\leq & C (\|u_0\|_{\dot{H}^{3/2}} + \|u_1\|_{\dot{H}^{1/2}} + \|2 u^{(k-1)}_x \wedge u^{(k-1)}_y\|_{L^1_{[0,T]} \dot{H}^{1/2}}) \|2 u^{(k-1)}_x \wedge u^{(k-1)}_y - 2 F\|_{L^1_{[0,T]} \dot{H}^{1/2}} \\
& + C  \|2 u^{(k-1)}_x \wedge u^{(k-1)}_y - 2 F\|_{L^1_{[0,T]} \dot{H}^{1/2}} (\|u_0\|_{\dot{H}^{3/2}} + \|u_1\|_{\dot{H}^{1/2}} + \|2 F\|_{L^1_{[0,T]} \dot{H}^{1/2}}) \\
\leq &2 C  (K + 2T^{1/2} A)  T^{1/2} \| 2 u^{(k-1)}_x \wedge u^{(k-1)}_y - 2 F \|_{L^2_{[0,T]} \dot{H}^{1/2}} \\
\to& 0 
\end{align*}
as $k \to \infty$. (The second inequality follows from (\ref{eq:induct2}) and (\ref{eq:FL2bdd}), and the last convergence follows from our definition of $F$.) This proves our claim, and hence our existence result.

Next, for uniqueness, assume that $u$, $v \colon [0,T] \times \mathbb{R}^2 \to \mathbb{R}^3$ solves
$$
\square u = 2 u_x \wedge u_y, \quad \left.u \right|_{t=0} = u_0, \quad \left.\partial_t u\right|_{t=0} = u_1
$$
$$
\square v = 2 v_x \wedge v_y, \quad \left.v \right|_{t=0} = u_0, \quad \left.\partial_t v\right|_{t=0} = u_1
$$
with
$$
u, v \in C^0_{[0,T]} \dot{H}^{3/2}, \quad \partial_t u, \partial_t v \in C^0_{[0,T]} \dot{H}^{1/2},
$$
and
$$
\|u_x \wedge u_y\|_{L^2_{[0,T]} \dot{H}^{1/2}} \leq A , \quad \|v_x \wedge v_y\|_{L^2_{[0,T]} \dot{H}^{1/2}} \leq A.
$$
We will show $$u = v \quad \text{on $[0,T]$}.$$

To do so, note that
$$
\square (u-v) = 2 u_x \wedge u_y - 2 v_x \wedge v_y, \quad \left.(u-v) \right|_{t=0} = 0, \quad \left.\partial_t (u-v) \right|_{t=0} = 0.
$$
So by the energy estimate (\ref{eq:energy32}),
\begin{align*}
\|u-v\|_{C^0_{[0,T]} \dot{H}^{3/2}} + \|\partial_t (u-v)\|_{C^0_{[0,T]} \dot{H}^{1/2}}
\leq & C T^{1/2} \|2 u_x \wedge u_y - 2 v_x \wedge v_y\|_{L^2_{[0,T]} \dot{H}^{1/2}}.
\end{align*}
Now by (\ref{eq:wedgeest}),
\begin{align*}
& \|2 u_x \wedge u_y - 2 v_x \wedge v_y\|_{L^2_{[0,T]} \dot{H}^{1/2}} \\
\leq & 2( \|u_x \wedge (u-v)_y\|_{L^2_{[0,T]} \dot{H}^{1/2}} +   \|(u-v)_x \wedge v_y\|_{L^2_{[0,T]} \dot{H}^{1/2}} )\\
\leq &2C( \|u_0\|_{\dot{H}^{3/2}} + \|u_1\|_{\dot{H}^{1/2}} + \|2u_x \wedge u_y\|_{L^1_{[0,T]} \dot{H}^{1/2}})  \|2u_x \wedge u_y - 2v_x \wedge v_y\|_{L^1_{[0,T]} \dot{H}^{1/2}}  \\
&\quad + 2C \|2u_x \wedge u_y - 2v_x \wedge v_y\|_{L^1_{[0,T]} \dot{H}^{1/2}}  ( \|u_0\|_{\dot{H}^{3/2}} + \|u_1\|_{\dot{H}^{1/2}} + \|2v_x \wedge v_y\|_{L^1_{[0,T]} \dot{H}^{1/2}})   \\
\leq & 4CT^{1/2} (K + 2T^{1/2}A) \|2u_x \wedge u_y - 2v_x \wedge v_y\|_{L^2_{[0,T]} \dot{H}^{1/2}} \\
\leq & \frac{1}{2}   \|2u_x \wedge u_y - 2v_x \wedge v_y\|_{L^2_{[0,T]} \dot{H}^{1/2}}
\end{align*}
by (\ref{eq:Tchoice3}). Thus
$$
\|2 u_x \wedge u_y - 2 v_x \wedge v_y\|_{L^2_{[0,T]} \dot{H}^{1/2}} = 0,
$$
which implies
$$
\|u-v\|_{C^0_{[0,T]} \dot{H}^{3/2}} + \|\partial_t (u-v)\|_{C^0_{[0,T]} \dot{H}^{1/2}} = 0.
$$
So $u = v$ on $[0,T]$, as desired.

Finally, we prove the continuous dependence of the solution on initial data. Suppose $(u_0,u_1)$ and $(v_0,v_1)$ are initial data, so that
$$
\|u_0\|_{\dot{H}^{3/2}} + \|u_1\|_{\dot{H}^{1/2}} \leq K, \quad
\|v_0\|_{\dot{H}^{3/2}} + \|v_1\|_{\dot{H}^{1/2}} \leq K, 
$$
and
$$
\|u_0 - v_0\|_{\dot{H}^{3/2}} + \|u_1 - v_1 \|_{\dot{H}^{1/2}} \leq \varepsilon.
$$
Let $u$, $v$ be the unique solution to
$$
\square u = 2 u_x \wedge u_y, \quad \left.u \right|_{t=0} = u_0, \quad \left.\partial_t u\right|_{t=0} = u_1
$$
$$
\square v = 2 v_x \wedge v_y, \quad \left.v \right|_{t=0} = v_0, \quad \left.\partial_t v\right|_{t=0} = v_1
$$
with
$$
u,v \in C^0_{[0,T]} \dot{H}^{3/2}, \quad \partial_t u, \partial_t v \in C^0_{[0,T]} \dot{H}^{1/2},
$$
and
$$
\|u_x \wedge u_y\|_{L^2_{[0,T]} \dot{H}^{1/2}} \leq A, \quad \|v_x \wedge v_y\|_{L^2_{[0,T]} \dot{H}^{1/2}} \leq A.
$$
We claim
\begin{equation} \label{eq:ctsdepend}
\|u-v\|_{C^0_{[0,T]} \dot{H}^{3/2}} + \|\partial_t u - \partial_t v\|_{C^0_{[0,T]} \dot{H}^{1/2}} \leq B \varepsilon
\end{equation}
where
$$
B = \max\{4, 4C(K+2T^{1/2}A)\}.
$$
This will prove continuous dependence on initial data.

To see this, recall that $u$, $v$ are the limits in our space $X$ of a sequence $u^{(k)}$, $v^{(k)}$ respectively, where $u^{(0)} = v^{(0)} = 0$, and
$$
\square u^{(k+1)} = 2 u^{(k)}_x \wedge u^{(k)}_y, \quad \left.u^{(k+1)} \right|_{t=0} = u_0, \quad \left.\partial_t u^{(k+1)} \right|_{t=0} = u_1,
$$
$$
\square v^{(k+1)} = 2 v^{(k)}_x \wedge v^{(k)}_y, \quad \left.v^{(k+1)} \right|_{t=0} = v_0, \quad \left.\partial_t v^{(k+1)}\right|_{t=0} = v_1
$$
for all $k \geq 0$.
We will prove, by induction, that
\begin{equation} \label{eq:induct5}
\|u^{(k)}_x \wedge u^{(k)}_y -v^{(k)}_x \wedge v^{(k)}_y\|_{L^2_{[0,T]} \dot{H}^{1/2}} \leq B \varepsilon
\end{equation}
for all $k \geq 0$. Assuming this for the moment. Then the energy estimate (\ref{eq:energy32}) shows that for all $k \geq 0$,
\begin{align*}
& \|u^{(k+1)}-v^{(k+1)}\|_{C^0_{[0,T]} \dot{H}^{3/2}} + \|\partial_t u^{(k+1)} - \partial_t v^{(k+1)}\|_{C^0_{[0,T]} \dot{H}^{1/2}} \\
\leq & 2 ( \|u_0-v_0\|_{\dot{H}^{3/2}} + \|u_1-v_1\|_{\dot{H}^{1/2}} + \|2u^{(k)}_x \wedge u^{(k)}_y - 2v^{(k)}_x \wedge v^{(k)}_y \|_{L^1_{[0,T]} \dot{H}^{1/2}}) \\
\leq & 2 (\varepsilon + 2T^{1/2} B \varepsilon) \\
\leq & 2 \varepsilon + \frac{B}{2} \varepsilon \\
\leq & B \varepsilon.
\end{align*}
(The third-to-last inequality follows from (\ref{eq:induct5}), and the second-to-last follows from (\ref{eq:Tchoice1}). The last inequality follows from our choice of $B$ that $B \geq 4$.) Letting $k \to \infty$, (\ref{eq:ctsdepend}) follows.

Thus it remains to prove (\ref{eq:induct5}). It clearly holds when $k = 0$. Now suppose $k \geq 1$. 
By (\ref{eq:wedgeest}),
\begin{align*}
&\|u^{(k)}_x \wedge u^{(k)}_y -v^{(k)}_x \wedge v^{(k)}_y\|_{L^2_{[0,T]} \dot{H}^{1/2}} \\
\leq &\|u^{(k)}_x \wedge (u^{(k)}-v^{(k)})_y\|_{L^2_{[0,T]} \dot{H}^{1/2}} + \|(u^{(k)}-v^{(k)})_x \wedge v^{(k)}_y\|_{L^2_{[0,T]} \dot{H}^{1/2}} \\
\leq & C (\|u_0\|_{\dot{H}^{3/2}} + \|u_1\|_{\dot{H}^{1/2}} + \|2u^{(k-1)}_x \wedge u^{(k-1)}_y\|_{L^1_{[0,T]} \dot{H}^{1/2}}) \\
&\qquad \qquad \cdot  (\|u_0-v_0\|_{\dot{H}^{3/2}} + \|u_1-v_1\|_{\dot{H}^{1/2}} + \|2 u^{(k-1)}_x \wedge u^{(k-1)}_y - 2 v^{(k-1)}_x \wedge v^{(k-1)}_y \|_{L^1_{[0,T]} \dot{H}^{1/2}} ) \\
&+ C (\|u_0-v_0\|_{\dot{H}^{3/2}} + \|u_1-v_1\|_{\dot{H}^{1/2}} + \|2 u^{(k-1)}_x \wedge u^{(k-1)}_y - 2 v^{(k-1)}_x \wedge v^{(k-1)}_y \|_{L^1_{[0,T]} \dot{H}^{1/2}} ) \\
&\qquad \qquad \cdot (\|v_0\|_{\dot{H}^{3/2}} + \|v_1\|_{\dot{H}^{1/2}} + \|2v^{(k-1)}_x \wedge v^{(k-1)}_y\|_{L^1_{[0,T]} \dot{H}^{1/2}})
\end{align*}
By our choice of initial data $(u_0,u_1)$ and $(v_0,v_1)$, and using (\ref{eq:induct2}) with our induction hypothesis (\ref{eq:induct5}) with $k$ replaced by $k-1$, this is bounded by
\begin{align*}
& 2C (K + 2T^{1/2} A)  (\varepsilon + 2 T^{1/2} B \varepsilon),
\end{align*}
which by choice of $B$ (so that $2C(K+2T^{1/2}A) \leq \frac{B}{2}$) and (\ref{eq:Tchoice3}) is bounded by
\begin{align*}
\frac{B}{2} \varepsilon + \frac{B}{2} \varepsilon = B \varepsilon.
\end{align*}
This completes our induction, and hence the proof of our theorem.
\end{proof}

We briefly outline the proof of (\ref{eq:derkeyQ12++}). The space-time Fourier transform of $\phi_+ \cdot \psi_+$ is the convolution of $\tphi_+$ with $\tpsi_+$. We compute this convolution by testing it against a test function $\varphi(\xi,\tau)$:
\begin{align*}
&\int_{\mathbb{R}} \int_{\mathbb{R}^2} (\tphi_+ * \tpsi_+)(\xi,\tau) \varphi(\xi,\tau) d\xi d\tau \\
=& \int_{\mathbb{R}} \int_{\mathbb{R}} \int_{\mathbb{R}^2} \int_{\mathbb{R}^2} \tphi_+(\xi, \tau) \tpsi_+(\xi', \tau') \varphi(\xi + \xi', \tau + \tau') d\xi d\xi' d\tau d\tau' \\
=& \int_{\mathbb{R}^2} \int_{\mathbb{R}^2} \frac{\hat{f}(\xi)}{|\xi|} \frac{\hat{g}(\xi')}{|\xi'|} \varphi(\xi+\xi', |\xi| + |\xi'|) d\xi d\xi' \\
=& \int_{\mathbb{R}^2} \int_{\mathbb{R}^2} \frac{\hat{f}(\xi-\xi')}{|\xi-\xi'|} \frac{\hat{g}(\xi')}{|\xi'|} \varphi(\xi, |\xi-\xi'| + |\xi'|) d\xi d\xi'. 
\end{align*}
Now write $\xi' $ in polar coordinates: $\xi' = \rho \omega$ where $\rho > 0$ and $\omega \in \mathbb{S}^1$. Then the above  integral becomes
\begin{equation} \label{eq:convpol}
\int_{\mathbb{R}^2}  \int_{\mathbb{S}^1} \int_0^{\infty} \frac{\hat{f}(\xi-\rho \omega)}{|\xi-\rho \omega|} \hat{g}(\rho \omega) \varphi(\xi, |\xi-\rho \omega| + \rho) d\rho d\omega d\xi.
\end{equation}
We change variables from $\rho$ to $\tau$, where $\tau$ is a new variable defined as 
$$
\tau = |\xi - \rho \omega| + \rho;
$$
then
\begin{align*}
\tau^2 
&= |\xi - \rho \omega|^2 + \rho^2 + 2 \rho |\xi-\rho \omega| \\
&= |\xi|^2 - 2\rho \xi \cdot \omega + \rho^2 + \rho^2 + 2 \rho (\tau - \rho) \\
&= |\xi|^2 + 2\rho (\tau - \xi \cdot \omega),
\end{align*}
which implies 
\begin{equation} \label{eq:rhodef}
\rho = \rho(\xi,\tau,\omega) = \frac{\tau^2 - |\xi|^2}{2(\tau - \xi \cdot \omega)}.
\end{equation}
(Incidentally, this shows the change of variables is legitimate; it is a (smooth) bijection of $\rho \in [0,\infty)$, to $\tau \in [|\xi|, \infty)$.) Hence (\ref{eq:convpol}) becomes
$$
\int_{\mathbb{R}^2}  \int_{\mathbb{R}} \int_{\mathbb{S}^1} \chi_{\tau > |\xi|} \frac{\hat{f}(\xi-\rho \omega)}{|\xi-\rho \omega|} \hat{g}(\rho \omega) \frac{\partial \rho}{\partial \tau}  \varphi(\xi, \tau) d\omega d\tau d\xi,
$$
(henceforth $\rho$ will be defined by $\xi, \tau, \omega$ as in (\ref{eq:rhodef})). This shows the convolution $\tphi_+ * \tpsi_+$ is given by
\begin{equation} \label{eq:++conv1}
\tphi_+ * \tpsi_+ (\xi, \tau) =  \chi_{\tau > |\xi|}  \int_{\mathbb{S}^1} \frac{\hat{f}(\xi-\rho \omega)}{|\xi-\rho \omega|} \hat{g}(\rho \omega) \frac{\partial \rho}{\partial \tau}  d\omega.
\end{equation}
We further simplify this formula: 
$$
|\xi - \rho \omega| = \tau - \rho = \tau - \frac{\tau^2 - |\xi|^2}{2(\tau - \xi \cdot \omega)} = \frac{\tau^2 - 2\tau \xi \cdot \omega + |\xi|^2}{2(\tau - \xi \cdot \omega)}.
$$
Also,
\begin{equation} \label{eq:drhodtau}
\frac{\partial \rho}{\partial \tau} = \frac{(\tau - \xi \cdot \omega)(2\tau) - (\tau^2 - |\xi|^2)}{2(\tau - \xi \cdot \omega)^2} = \frac{\tau^2 - 2\tau \xi \cdot \omega + |\xi|^2}{2(\tau - \xi \cdot \omega)^2}.
\end{equation}
Hence
$$
\frac{1}{|\xi-\rho \cdot \omega|} \frac{\partial \rho}{\partial \tau} = \frac{1}{\tau - \xi \cdot \omega} = \frac{2\rho}{\tau^2-|\xi|^2},
$$
and (\ref{eq:++conv1}) becomes
\begin{equation} \label{eq:++conv2}
\tphi_+ * \tpsi_+ (\xi, \tau) =  \frac{2\chi_{\tau > |\xi|}}{\tau^2-|\xi|^2} \int_{\mathbb{S}^1}  \hat{f}(\xi-\rho \omega) \hat{g}(\rho \omega) \rho d\omega.
\end{equation}

Now to prove (\ref{eq:derkeyQ12++}), note that
\begin{align*}
& |\mathcal{F}((-\Delta)^{1/4} Q_{12}(\phi_+,\psi_+))(\xi,\tau)| 
= |\xi|^{1/2} |\mathcal{F}(Q_{12}(\phi_+,\psi_+))(\xi,\tau)| 
\end{align*}
and (at least formally)
\begin{align*}
& \mathcal{F}(Q_{12}(\phi_+,\psi_+))(\xi,\tau) \\
= &\int_{\mathbb{R}} \int_{\mathbb{R}^2} (\xi_1 \xi_2' - \xi_2 \xi_1') \tphi_+(\xi-\xi',\tau-\tau') \tpsi_+(\xi',\tau') d\xi' d\tau'.
\end{align*}
So putting absolute values inside the integral, and using
$$
|\xi|^{1/2} \leq |\xi-\xi'|^{1/2} + |\xi'|^{1/2},
$$
we see that one has
\begin{align*}
& |\mathcal{F}((-\Delta)^{1/4} Q_{12}(\phi_+,\psi_+))(\xi,\tau)| \\
\leq &\int_{\mathbb{R}} \int_{\mathbb{R}^2} |\xi_1 \xi_2' - \xi_2 \xi_1'| |\xi-\xi'|^{1/2} |\tphi_+(\xi-\xi',\tau-\tau')| |\tpsi_+(\xi',\tau')| d\xi' d\tau' \\
& \quad + \int_{\mathbb{R}} \int_{\mathbb{R}^2} |\xi_1 \xi_2' - \xi_2 \xi_1'| |\tphi_+(\xi-\xi',\tau-\tau')| |\xi'|^{1/2} |\tpsi_+(\xi',\tau')| d\xi' d\tau' \\
=& I + II
\end{align*}
The integral $II$ can be brought, via a change of variables $\xi' \mapsto \xi-\xi'$, $\tau' \mapsto \tau - \tau'$, into
$$
\int_{\mathbb{R}} \int_{\mathbb{R}^2} |\xi_1 \xi_2' - \xi_2 \xi_1'| |\xi-\xi'|^{1/2} |\tpsi_+(\xi-\xi',\tau-\tau')| |\tphi_+(\xi',\tau')| d\xi' d\tau'
$$
which is the same as integral $I$, except now the roles of $\phi_+$ and $\psi_+$ (hence the roles of $f$ and $g$) are reversed. Since the right hand side of our desired estimate (\ref{eq:derkeyQ12++}) is symmetric in $f$ and $g$, it suffices now to bound the integral $I$. But $I$ can be computed by testing against a test function as above. We then get, in a similar manner that we derived (\ref{eq:++conv2}), that
\begin{align*}
I(\xi,\tau)
=&  \frac{2\chi_{\tau > |\xi|}}{\tau^2 - |\xi|^2}  \int_{\mathbb{S}^1} \rho |\xi_1 \omega_2 - \xi_2 \omega_1|  |\hat{F}(\xi-\rho \omega)| |\hat{g}(\rho \omega)| \rho d\omega
\end{align*}
where 
$$
F:= (-\Delta)^{1/4} f.
$$
It follows that
\begin{align*}
|I(\xi,\tau)|^2 &\lesssim \frac{\chi_{\tau > |\xi|}}{(\tau^2-|\xi|^2)^2} \int_{\mathbb{S}^1} \rho^2 |\xi_1 \omega_2 - \xi_2 \omega_1|^2 |\hat{F}(\xi-\rho \omega)|^2 |\hat{g}(\rho \omega)|^2 \rho^2 d\omega.
\end{align*}
We now integrate with respect to $\xi$ and $\tau$, use Fubini's theorem, and change variables $\tau \mapsto \rho$. Then
\begin{align}
&\int_{\mathbb{R}^2} \int_{\mathbb{R}} |I(\xi,\tau)|^2 d\tau d\xi \notag \\
\lesssim& \int_{\mathbb{R}^2} \int_{\mathbb{S}^1} \int_{\rho = 0}^{\rho = \infty} \frac{\rho^2 |\xi_1 \omega_2 - \xi_2 \omega_1|^2}{(\tau^2-|\xi|^2)^2} \frac{\partial \tau}{\partial \rho}  |\hat{F}(\xi-\rho \omega)|^2 |\hat{g}(\rho \omega)|^2 \rho^2 d\rho d\omega d\xi. \label{eq:intI2}
\end{align}
But we claim
\begin{equation} \label{eq:keyQ12++}
\frac{\rho^2 (\xi_1 \omega_2 - \xi_2 \omega_1 )^2}{(\tau^2-|\xi|^2)^2}  \frac{\partial \tau}{\partial \rho} \leq \frac{1}{2}.
\end{equation}
In fact, remembering $\frac{\partial \tau}{\partial \rho}$ is the reciprocal of $\frac{\partial \rho}{\partial \tau}$, which we computed in (\ref{eq:drhodtau}), we have
\begin{align*}
\frac{\rho^2 (\xi_1 \omega_2 - \xi_2 \omega_1 )^2}{(\tau^2-|\xi|^2)^2}  \frac{\partial \tau}{\partial \rho} 
=& \frac{\rho^2 (\xi_1 \omega_2 - \xi_2 \omega_1 )^2 [2 (\tau - \xi \cdot \omega)^2]}{(\tau^2-|\xi|^2)^2 (\tau^2 - 2 \tau \xi \cdot \omega + |\xi|^2)} \\
=& \frac{(\xi_1 \omega_2 - \xi_2 \omega_1 )^2}{2(\tau^2 - 2 \tau \xi \cdot \omega + |\xi|^2)}
\end{align*}
But the denominator can be simplified, by writing
\begin{align*}
\tau^2 - 2 \tau \xi \cdot \omega + |\xi|^2
&= (\tau - \xi \cdot \omega)^2 + |\xi|^2 - (\xi \cdot \omega)^2 \\
&= (\tau - \xi \cdot \omega)^2 + (\xi_1 \omega_2 - \xi_2 \omega_1)^2
\end{align*}
(the last equality holds since we are in 2-dimensions). Thus we see that the above quotient is bounded by $1/2$, as was claimed in (\ref{eq:keyQ12++}). 

Now by (\ref{eq:intI2}) and (\ref{eq:keyQ12++}), we see that
\begin{align*}
\int_{\mathbb{R}^2} \int_{\mathbb{R}} |I(\xi,\tau)|^2 d\tau d\xi 
\leq &C \int_{\mathbb{R}^2} \int_{\mathbb{S}^1} \int_{\rho = 0}^{\rho = \infty}  |\hat{F}(\xi-\rho \omega)|^2 |\hat{g}(\rho \omega)|^2 \rho^2 d\rho d\omega d\xi \\
=& C \|F\|_{L^2}^2 \|g\|_{\dot{H}^{1/2}}^2 \\
=& C \|f\|_{\dot{H}^{1/2}}^2 \|g\|_{\dot{H}^{1/2}}^2.
\end{align*}
This completes the proof of (\ref{eq:derkeyQ12++}).


\begin{bibdiv}
\begin{biblist}

\bib{MR733715}{article}{
   author={Brezis, Ha{\"{\i}}m},
   author={Coron, Jean-Michel},
   title={Multiple solutions of $H$-systems and Rellich's conjecture},
   journal={Comm. Pure Appl. Math.},
   volume={37},
   date={1984},
   number={2},
   pages={149--187},
}

\bib{MR784102}{article}{
   author={Brezis, Ha{\"{\i}}m},
   author={Coron, Jean-Michel},
   title={Convergence of solutions of $H$-systems or how to blow bubbles},
   journal={Arch. Rational Mech. Anal.},
   volume={89},
   date={1985},
   number={1},
   pages={21--56},
}

\bib{MR2221202}{article}{
   author={Caldiroli, Paolo},
   author={Musina, Roberta},
   title={The Dirichlet problem for H-systems with small boundary
     data: blow-up phenomena and nonexistence results},
   journal={Arch. Rational Mech. Anal.},
   volume={181},
   date={2006},
   number={1},
   pages={1--42},
}


\bib{MR2154671}{article}{
   author={Chanillo, Sagun},
   author={Malchiodi, Andrea},
   title={Asymptotic Morse theory for the equation $\Delta v = 2 v_x \wedge v_y$},
   journal={Communications in Analysis and Geometry},
   volume={13},
   date={2005},
   number={1},
   pages={187--251}
}

\bib{MR3010056}{article}{
   author={Chanillo, Sagun},
   author={Yung, Po-Lam},
   title={Wave equations associated to Liouville systems and constant mean
   curvature equations},
   journal={Adv. Math.},
   volume={235},
   date={2013},
   pages={187--207},
}
		



\bib{MR0058082}{article}{
   author={Hartman, Philip},
   author={Wintner, Aurel},
   title={On the local behavior of solutions of non-parabolic partial
   differential equations},
   journal={Amer. J. Math.},
   volume={75},
   date={1953},
   pages={449--476},
}

\bib{MR0256276}{article}{
   author={Hildebrandt, Stefan},
   title={On the Plateau problem for surfaces of constant mean curvature},
   journal={Comm. Pure Appl. Math.},
   volume={23},
   date={1970},
   pages={97--114},
}

\bib{MR2461508}{article}{
   author={Kenig, Carlos E.},
   author={Merle, Frank},
   title={Global well-posedness, scattering and blow-up for the
   energy-critical focusing non-linear wave equation},
   journal={Acta Math.},
   volume={201},
   date={2008},
   number={2},
   pages={147--212},
}

\bib{MR1231427}{article}{
   author={Klainerman, Sergiu},
   author={Machedon, Matei},
   title={Space-time estimates for null forms and the local existence
   theorem},
   journal={Comm. Pure Appl. Math.},
   volume={46},
   date={1993},
   number={9},
   pages={1221--1268},
}

\bib{MR1381973}{article}{
   author={Klainerman, Sergiu},
   author={Machedon, Matei},
   title={Smoothing estimates for null forms and applications},
   note={A celebration of John F. Nash, Jr.},
   journal={Duke Math. J.},
   volume={81},
   date={1995},
   number={1},
   pages={99--133 (1996)},
}

\bib{MR1446618}{article}{
   author={Klainerman, Sergiu},
   author={Machedon, Matei},
   title={On the regularity properties of a model problem related to wave
   maps},
   journal={Duke Math. J.},
   volume={87},
   date={1997},
   number={3},
   pages={553--589},
}

\bib{MR1452172}{article}{
   author={Klainerman, Sergiu},
   author={Selberg, Sigmund},
   title={Remark on the optimal regularity for equations of wave maps type},
   journal={Comm. Partial Differential Equations},
   volume={22},
   date={1997},
   number={5-6},
   pages={901--918},
}

\bib{MR1901147}{article}{
   author={Klainerman, Sergiu},
   author={Selberg, Sigmund},
   title={Bilinear estimates and applications to nonlinear wave equations},
   journal={Commun. Contemp. Math.},
   volume={4},
   date={2002},
   number={2},
   pages={223--295},
}


\bib{MR1990880}{article}{
   author={Krieger, Joachim},
   title={Global regularity of wave maps from ${\bf R}^{3+1}$ to
   surfaces},
   journal={Comm. Math. Phys.},
   volume={238},
   date={2003},
   number={1-2},
   pages={333--366},
}

\bib{MR2094472}{article}{
   author={Krieger, Joachim},
   title={Global regularity of wave maps from $\bold R^{2+1}$ to $H^2$. Small energy},
   journal={Comm. Math. Phys.},
   volume={250},
   date={2004},
   number={3},
   pages={507--580},
}

\bib{MR2895939}{book}{
   author={Krieger, Joachim},
   author={Schlag, Wilhelm},
   title={Concentration compactness for critical wave maps},
   series={EMS Monographs in Mathematics},
   publisher={European Mathematical Society (EMS), Z\"urich},
   date={2012},
   pages={vi+484},
}

\bib{MR2372807}{article}{
   author={Krieger, J.},
   author={Schlag, W.},
   author={Tataru, D.},
   title={Renormalization and blow up for charge one equivariant critical
   wave maps},
   journal={Invent. Math.},
   volume={171},
   date={2008},
   number={3},
   pages={543--615},
}

\bib{MR2929728}{article}{
   author={Rapha{\"e}l, Pierre},
   author={Rodnianski, Igor},
   title={Stable blow up dynamics for the critical co-rotational wave maps
   and equivariant Yang-Mills problems},
   journal={Publ. Math. Inst. Hautes \'Etudes Sci.},
   date={2012},
   pages={1--122},
}

\bib{MR2680419}{article}{
   author={Rodnianski, Igor},
   author={Sterbenz, Jacob},
   title={On the formation of singularities in the critical ${\rm O}(3)$
   $\sigma$-model},
   journal={Ann. of Math. (2)},
   volume={172},
   date={2010},
   number={1},
   pages={187--242},
}

\bib{MR1674843}{book}{
   author={Shatah, Jalal},
   author={Struwe, Michael},
   title={Geometric wave equations},
   series={Courant Lecture Notes in Mathematics},
   volume={2},
   publisher={New York University, Courant Institute of Mathematical
   Sciences, New York; American Mathematical Society, Providence, RI},
   date={1998},
   pages={viii+153},
}

\bib{MR2657818}{article}{
   author={Sterbenz, Jacob},
   author={Tataru, Daniel},
   title={Regularity of wave-maps in dimension $2+1$},
   journal={Comm. Math. Phys.},
   volume={298},
   date={2010},
   number={1},
   pages={231--264},
}

\bib{MR2657817}{article}{
   author={Sterbenz, Jacob},
   author={Tataru, Daniel},
   title={Energy dispersed large data wave maps in $2+1$ dimensions},
   journal={Comm. Math. Phys.},
   volume={298},
   date={2010},
   number={1},
   pages={139--230},
}

\bib{MR823116}{article}{
   author={Struwe, Michael},
   title={Nonuniqueness in the Plateau problem for surfaces of constant mean
   curvature},
   journal={Arch. Rational Mech. Anal.},
   volume={93},
   date={1986},
   number={2},
   pages={135--157},
}

\bib{MR774369}{article}{
   author={Struwe, Michael},
   title={Large $H$-surfaces via the mountain-pass-lemma},
   journal={Math. Ann.},
   volume={270},
   date={1985},
   number={3},
   pages={441--459},
}

\bib{MR926524}{article}{
   author={Struwe, Michael},
   title={The existence of surfaces of constant mean curvature with free
   boundaries},
   journal={Acta Math.},
   volume={160},
   date={1988},
   number={1-2},
   pages={19--64},
}

\bib{MR1820329}{article}{
   author={Tao, Terence},
   title={Global regularity of wave maps. I. Small critical Sobolev norm in
   high dimension},
   journal={Internat. Math. Res. Notices},
   date={2001},
   number={6},
   pages={299--328},
}
	
\bib{MR1869874}{article}{
   author={Tao, Terence},
   title={Global regularity of wave maps. II. Small energy in two
   dimensions},
   journal={Comm. Math. Phys.},
   volume={224},
   date={2001},
   number={2},
   pages={443--544},
}

\bib{MR1641721}{article}{
   author={Tataru, Daniel},
   title={Local and global results for wave maps. I},
   journal={Comm. Partial Differential Equations},
   volume={23},
   date={1998},
   number={9-10},
   pages={1781--1793},
}

\bib{MR1827277}{article}{
   author={Tataru, Daniel},
   title={On global existence and scattering for the wave maps equation},
   journal={Amer. J. Math.},
   volume={123},
   date={2001},
   number={1},
   pages={37--77},
}

\bib{MR2130618}{article}{
   author={Tataru, Daniel},
   title={Rough solutions for the wave maps equation},
   journal={Amer. J. Math.},
   volume={127},
   date={2005},
   number={2},
   pages={293--377},
}

\bib{MR0243467}{article}{
   author={Wente, Henry C.},
   title={An existence theorem for surfaces of constant mean curvature},
   journal={J. Math. Anal. Appl.},
   volume={26},
   date={1969},
   pages={318--344},
}

\bib{MR0374673}{article}{
   author={Wente, Henry C.},
   title={The differential equation $\Delta x=2H(x_{u}\wedge x_{v})$
   with vanishing boundary values},
   journal={Proc. Amer. Math. Soc.},
   volume={50},
   date={1975},
   pages={131--137},
}

\end{biblist}
\end{bibdiv}
\end{document}